\documentclass[10pt]{article}

\usepackage{mathptmx}
\usepackage{wrapfig}
\usepackage{eucal}
\usepackage{amsmath}
\usepackage{amscd}
\usepackage{amssymb}
\usepackage{amsthm}
\usepackage{xspace}
\usepackage[all,tips]{xy}
\usepackage[dvips]{graphicx}
\usepackage{graphics}
\usepackage{epsfig}
\usepackage{verbatim}
\usepackage{syntonly}
\usepackage{hyperref}
\usepackage{amssymb}
\usepackage{color}
\usepackage{url}
\usepackage{overpic}
\usepackage{wrapfig}

\newtheorem{theorem}{Theorem}
\newtheorem{proposition}[theorem]{Proposition}

\newtheorem{corollary}[theorem]{Corollary}
\newtheorem{lemma}[theorem]{Lemma}

\input xy
\xyoption{all}

\DeclareMathOperator{\GL}{GL}

\newcommand{\Sym}{\ensuremath{\text{Sym}}}

\usepackage{fancyhdr}

\fancypagestyle{plain}

\begin{document}

%-----------------------------------------------------------
%-----------------------------------------------------------

\title{\textbf{Local commensurability graphs \\ of solvable groups}}

\author{Khalid Bou-Rabee\thanks{K.B. supported in part by NSF grant
    DMS-1405609 and Cycle 46 PSC-CUNY Research Award} \and Chen Shi}

\maketitle

%-----------------------------------------------------------
%-----------------------------------------------------------

\begin{abstract}
The \emph{commensurability index} between two subgroups $A, B$ of a group $G$ is $[A : A \cap B] [B : A\cap B]$. This gives a notion of distance amongst finite-index subgroups of $G$, which is encoded in the \emph{p-local commensurability graphs} of $G$. We show that for any metabelian group, any component of the $p$-local commensurabilty graph of $G$ has diameter bounded above by 4.
However, no universal upper bound on diameters of components exists for the class of finite solvable groups. In the appendix we give a complete classification of components for upper triangular matrix groups in $\GL(2, \mathbb{F}_q)$.
\end{abstract}

%Recall that the {\bf $p$-local subgroup growth function} of a group $G$ assigns to every number $p^{k}$, the number of subgroups of index $p^\ell \leq p^{k}$. 
Let $G$ be a group and $p$ a prime number.
The \emph{$p$-local commensurability graph} of $G$ has all finite-index subgroups of $G$ as vertices, and edges are drawn between $A$ and $B$ if $[A: A \cap B][B:A \cap B]$ is a power of $p$. 
This graph was first introduced by K.B. and D.~Studenmund in \cite{BD15} where they proved results regarding the diameters of components.
The \emph{connected diameter} of a graph is the supremum over all graph diameters of every component of the graph. 
 In \cite{BD15}, it was shown that every $p$-local commensurability graph of any nilpotent group has connected diameter less than or equal to one. 
 In contrast, they also showed that every nonabelian free group has connected diameter equal to infinity for all of their $p$-local commensurability graphs. 
Here we show that solvable groups, in a sense, fill in the space between free groups and nilpotent groups.

Our first result shows that metabelian groups, like nilpotent groups, have a uniform bound on the connected diameters of their $p$-local commensurability graphs. This behavior is peculiar since many group invariants (e.g., subgroup growth \cite{MR1978431}, word growth \cite{MR1751083}, residual girth growth \cite{BC14, BD14}) do not distinguish metabelian groups in the class of non-nilpotent solvable groups.

\begin{theorem} \label{thm:mainB}
The connected diameter of the $p$-local commensurability graph of any metabelian group is at most $4$.
\end{theorem}

Theorem \ref{thm:mainB} follows quickly from the following more technical result. Theorem \ref{thm:mainA}, for instance, demonstrates that there exists many solvable groups that are not metabelian yet have uniform bounds on the connected diameters of their $p$-local commensurability graphs (e.g., the group of upper triangular matrices in $\GL(n, \mathbb{F}_q)$, for $n > 1$ and $q$ a prime).
The proofs of Theorems \ref{thm:mainB} and \ref{thm:mainA} appear in \S \ref{sec:thmA}.

\begin{theorem} \label{thm:mainA}
Let $G$ be a finite group. If a $p$-Sylow subgroup of $[G,G]$ is normal in $G$, then the connected diameter of the $p$-local commensurability graph of $G$ is at most 4.
\end{theorem}

Our next major result shows that the above theorems cannot be extended to arbitrary solvable groups. We prove this in \S \ref{sec:thmC}.

\begin{theorem} \label{thm:mainC}
For any $D > 0$, there exists a finite solvable group whose $3$-local commensurability graph has connected diameter greater than $D$.
\end{theorem}

It would be interesting to determine whether this result extends to other primes.
As a corollary of Theorem \ref{thm:mainC}, we arrive at a new proof of \cite[Theorem 2]{BD15} for the case $p = 3$. We note that in \cite{BD15} alternating groups are used in a significant way, while here we use solvable groups (which escapes the use of the classification of finite simple groups). The proof of the corollary also appears in  \S \ref{sec:thmC}.

\begin{corollary} \label{cor:main}
Let $F$ be a nonabelian free group.
The $3$-local commensurability graph of $F$ has infinite connected diameter.
\end{corollary}

In the appendix, we give an explicit description of the $p$-local commensurability graphs for the group of upper triangular matrices in $\GL_2(\mathbb{F}_q)$, as $p$ and $q$ varies over all primes.

We end the introduction with a brief history account. 
While the article, \cite{BD15}, began the study of local commensurability graphs with the goal of drawing fundamental group properties from residual invariants (a direction of much activity: \cite{KT2014}, \cite{BK12}, \cite{BM11}, \cite{KG2014}, \cite{BHP14}, \cite{BS13b}, \cite{KM12}, \cite{R12}, \cite{Patel:thesis}, \cite{MR1978431}), the study of graphs given by relations between subgroups goes back to work of B. Cs{\'a}k{\'a}ny and G. Poll{\'a}k in the 1960's \cite{MR0249328}. The $p$-local commensurability graphs are weighted pieces of what are called the \emph{intersection graphs} of a group (see, for instance,  \cite{MR3323326}). This graph has vertices consisting of proper subgroups where an edge is drawn between two groups if their intersection is non-trivial. The weights we add to these graphs significantly change the theory, and gives connections to A. Lubotzky and D. Segal's subgroup growth function (which, for every $n$, gives the number of subgroups of index $n$). Indeed, one can think of the theory of commensurability graphs as an interface between the theories of intersection graphs and subgroup growths.

\paragraph{Acknowledgements} The authors thank Michael Larsen and Daniel Studenmund for useful conversations. In particular, K. B. is grateful to Michael Larsen for suggesting the use of $p$-containment graphs that appear in the proof of Theorem \ref{thm:mainC}. 
Also, we are grateful to Daniel Studenmund for helpful corrections on an earlier draft.

\subsection*{Notation}

 Let $G$ be a group and $p$ be a prime number. We denote the $p$-local commensurability graph of $G$ by $\Gamma_p(G)$.
If $H$ and $M$ are two vertices in the same component of $\Gamma_{p}(G)$, then we say $H$ is \emph{$p$-connected} to $M$.
If $H$ and $M$ are adjacent in $\Gamma_p(G)$, then we say $H$ and $M$ are \emph{$p$-adjacent}. A path connecting $H$ and $M$ is called a \emph{$p$-path} and is denoted by $H-V_{1}-\dots-V_{m}-M$, where the \emph{length} of this path is defined to be $m+2$ and $H, V_{1}, \ldots, V_m, M$ are the vertices along the path.
 
\section{Metabelian groups: the proofs of Theorems \ref{thm:mainB} and \ref{thm:mainA}}
\label{sec:thmA}

We need some technical results for our proofs. In the following, we assume that $G$ is a finite group with derived series $G=G_{1}\rhd G_{2}\rhd \dots$
 given by $G_i = [G_{i-1}, G_{i-1}]$ for $i = 2, 3, \ldots$.

 \begin{lemma} \label{lem:VQandWQ}
 Let $Q$ be a normal $p$-subgroup of $G$.
 Let $V$ be a vertex of $\Gamma_{p}(G)$, then $V$ is p-adjacent to the group $VQ$.
 \end{lemma}
\begin{proof}
We need to show that $[V: V\cap VQ]$ and $ [VQ:V\cap VQ]$ are both powers of $p$. Clearly, we have $V\subset VQ$, then $[V:V\cap VQ]=[V:V]=1=p^{0}$.

Now, recalling a basic fact from algebra,
\begin{equation*}
|VQ|=\frac{|V||Q|}{|V\cap Q|},
\end{equation*}
gives
\begin{equation*}
[VQ:V\cap VQ]=[VQ:V]=\frac{|VQ|}{|V|}=\frac{\frac{|V||Q|}{|V\cap Q|}}{|V|}=\frac{|Q|}{|V\cap Q|}.
\end{equation*}
Since $Q$ is a $p$-group, $[VQ:V\cap VQ]$ is a power of $p$. It follows that $V$ is $p$-adjacent to $VQ$, as desired.
\end{proof}
\begin{lemma} \label{lem:padjup}
Let $Q$ be a normal $p$-subgroup of $G$.
Suppose $A,B < G$ and $A$ and $B$ are $p$-adjacent, then $AQ$ and $BQ$ are p-adjacent.
\end{lemma}
\begin{proof}
We need to show that $\frac{|AQ|}{|AQ\cap BQ|}$ and $\frac{|BQ|}{|AQ\cap BQ|}$ are both powers of $p$.
To this end, we compute 
\begin{equation*}
\frac{|AQ|}{|A\cap B|}=\frac{|A||Q|}{|A\cap B||A\cap Q|}.
\end{equation*}
Since $Q$ is a $p$-group, then $[Q:A\cap Q]=\frac{|Q|}{|A\cap Q|}$ is power of $p$.
Hence, $\frac{|AQ|}{|A\cap B|}$ is a power of $p$.
Also, $A\cap B<AQ\cap BQ<AQ$.
Then 
\begin{equation*}
[AQ:A\cap B]=[AQ:AQ\cap BQ][AQ\cap BQ:A\cap B]
\end{equation*}
 is power of $p$.
Therefore, $[AQ:AQ\cap BQ]$ is a power of $p$. By symmetry, $[BQ:AQ\cap BQ]$ is also a power of $p$.
\end{proof}

\begin{lemma} \label{lem:padjcap}
Let $V$ and $W$ be subgroups of $G$.
If $V$ is $p$-adjacent to $W$ and $N\triangleleft G$. Then $V\cap N$ is $p$-adjacent to $W\cap N$.
\end{lemma}
\begin{proof}
Let $\pi :G\rightarrow G/N$, then by \cite[Lemma 5]{BD15}, with $H=V$, $K=V\cap W$, and $N=V\cap N$, we get 
\begin{equation*}
[V:V\cap W]=[\pi(V):\pi(V\cap W)][V\cap N: V\cap W\cap N]
\end{equation*}
Since $V$ is p-adjacent to $W$, then $[V:V\cap W]$ is a power of $p$. Hence, $[V\cap N:V\cap W\cap N]$ is a power of $p$. Similarly, $[W\cap N:W\cap V\cap N]$ is a power of $p$.
\end{proof}
The previous lemma gives us a rigidity result:
\begin{lemma} \label{lem:final}
Let $i$ be a natural number.
Let $Q$ be a $p$-subgroup of $G$ that is normal in $G$ (then $VQ$ and $WQ$ are groups) where $Q \cap G_i$ is the $p$-Sylow subgroup of $G_i$.
If $VQ$ is $p$-connected to $WQ$, then $VQ\cap G_{i}=WQ\cap G_{i}$.
\end{lemma}
\begin{proof}
Since $Q \cap G_i$ is the $p$-Sylow subgroup of $G_{i}$, we have that $p$ does not divide $|G_{i}/(G_i \cap Q)|$. Thus, every connected component of $\Gamma_p(G_{i}/(G_i \cap Q))$ is a singleton (see Proposition \ref{prop:totaldisc}).
In particular, if $VQ\cap G_{i}$ is $p$-connected to $WQ\cap G_{i}$, then because both these groups contain $Q \cap G_i$, we have $VQ\cap G_{i}=WQ\cap G_{i}$.
To finish, note that $G_{i}\triangleleft G$ and $VQ$ is $p$-connected to $WQ$. Thus, $VQ\cap G_{i}$ is $p$-connected to $WQ\cap G_{i}$ by Lemma \ref{lem:padjcap}, which completes the proof.
\end{proof}
%
%\begin{lemma}
%If $G$ is a group, $\pi:G\rightarrow Q$ is a surjection, and $\gamma$ is a path in $\Gamma_{p}(G)$, then $\pi(\gamma)$ is a path in $\Gamma_{p}(Q)$ with length bounded above by the length of $\gamma$
%\end{lemma}
%
%\begin{proof}
%This is Lemma 7 from paper \cite{BD15}.
%\end{proof}
%
%\begin{lemma}
%Let $G$ and $H$ be groups and $f:G\rightarrow H$ be a surjective homomorphism, then $f([G,G])=[H,H]$.
%\end{lemma}
%\begin{proof}
%If $x\in [G,G]$, then $x$ has the form $x=g_{1}g_{2}g_{1}^{-1}g_{2}^{-1}\dots g_{k}g_{k+1}g_{k}^{-1}g_{k+1}^{-1}$. Then
%\begin{equation*}
%f(x)=f(g_{1}g_{2}g_{1}^{-1}g_{2}^{-1}\dots g_{k}g_{k+1}g_{k}^{-1}g_{k+1}^{-1})=f(g_{1})f(g_{2})f(g_{1}^{-1})f(g_{2}^{1})\dots f(g_{k})f(g_{k+1})f(g_{k}^{-1})f(g_{k+1}^{-1})
% \end{equation*}
% Which is in $H$.
% Conversely, since $f$ is surjective, then by a similar argument we will have if $y\in [H,H]$ the $f^{-1}(y)\in [G,G]$.
%\end{proof}

\begin{proof}[Proof of Theorem \ref{thm:mainA}]
Let $G$ be a finite group and $p$ a prime number. Let $Q$ be the $p$-Sylow subgroup of $G_2 = [G,G]$ (it is unique by assumption since all $p$-Sylow subgroups are conjugate).
Let $V$ and $W$ be finite-index subgroups of $G$ that are $p$-connected by a path in $\Gamma_{p}(G)$ such that $V=V_{1}-V_{2}-\dots-V_{m}=W$.
Now, by Lemma \ref{lem:VQandWQ} and Lemma \ref{lem:padjup} we obtain a new path in $\Gamma_{p}(G)$, 
\begin{equation*}
V_{1}-V_{1}Q-V_{2}Q-\dots-V_{m}Q-V_{m}.
\end{equation*}
Define 
$$A := \left<V_{1}Q,V_{2}Q,\dots, V_{m}Q \right> \:\:\:\text{ and }\:\:\: B:=A/(V_{1}Q\cap G_{2}).$$
Let $\pi :A\rightarrow B$ be the projection from $A$ to $B$.
By Lemma \ref{lem:final}, $\ker \pi = V_i Q \cap G_2$ for all $i = 1, \ldots, m$. Thus, $\ker \pi$ contains $[V_i Q, V_i Q]$.

Next, we claim that $\pi(V_{i}Q)$ is abelian for all $i = 1, \ldots, m$.
Indeed, if we restrict $\pi$ to $V_{i}Q$, we obtain a surjective homomorphism $f :V_{i}Q\rightarrow \pi(V_{i}Q)$, with
\begin{equation*}
f([V_{i}Q,V_{i}Q])=[\pi(V_{i}Q),\pi(V_{i}Q)].
\end{equation*}
By the previous paragraph, we have $[V_{i}Q,V_{i}Q]$ is in the kernel of $\pi$, and so $\pi(V_i Q)$ is abelian as claimed. 

Let $V_{i}^{'} :=\pi(V_{i}Q)$, then, as shown above, $V_{i}^{'}$ are abelian for all $i = 1 \ldots, m$. Let $\Pi$ be the set of all primes in the prime factorization of $|V_{1}^{'}|$ that are not $p$. Let $H$ be a Hall $\Pi$ subgroup of $V_{1}^{'}$. The group $V_{1}^{'}$ is abelian, so $H$ and $V_{1}^{'}\cap V_{2}^{'}$ are normal in $V_{1}^{'}$. Thus, by \cite[Lemma 6]{BD15}, $[V_{1}^{'}:V_{1}^{'}\cap V_{2}^{'}\cap H]$ is a power of $p$. Since $|H|$ is coprime with $p$ and
$$
[V_1' : V_1' \cap V_2' \cap H] = [V_1' : H][H: V_1' \cap V_2' \cap H],
$$
it follows that $V_{1}'\cap V_{2}'$ contains $H$, which means $V_{2}'$ contains H as a Hall $\Pi$ subgroup. Repeating this process, we get $H$ is a Hall $\Pi$-subgroup for $V_{i}'$, for all $i = 1,\ldots, m$. Thus, $H$ is of index a power of $p$ inside all $V_{i}'$.
In particular, we conclude that $\pi^{-1}(H)$ is $p$-adjacent to both $V_1 Q$ and $V_m Q$.
So we get a path $V_{1}-V_{1}Q-\pi^{-1}(H)-V_{m}Q-V_{m}$ of length 4, as desired.
\end{proof}

We are now ready to prove Theorem \ref{thm:mainB}.

\begin{proof}[Proof of Theorem \ref{thm:mainB}]
Let $G$ be a metabelian group and $p$ a prime number.
For any two vertices $V, W$ in a component of $\Gamma_p(G)$, let $N$ be the normal core of $V \cap W$.
Let $\pi: G \to G/N$ be the quotient map. Note that $G/N$ is a finite group.

The $p$-Sylow subgroup $Q$ of the abelian group, $[G/N,G/N]$, is unique. Hence, since $[G/N,G/N]$ is normal in $G/N$, it follows that $Q$ is normal in $G/N$. Then Theorem \ref{thm:mainA} applies to show that $\pi(V)$ is $p$-connected to $\pi(W)$ via a path of length at most 4 in $\Gamma_p(G/N)$.

\cite[Lemma 8, Part 1]{BD15} now applies to show that $V$ is connected to $W$ in $\Gamma_p(G)$ via a path of length at most 4. We are done since $V$ and $W$ were arbitrary vertices in $\Gamma_p(G)$, and so the connected diameter of $\Gamma_p(G)$ is at most 4.
\end{proof}

\section{Solvable groups: the proof of Theorem \ref{thm:mainC}}
\label{sec:thmC}

We introduce a slightly modified graph, which we call the \emph{$p$-containment graph} of a group $G$. 
This graph has vertices the finite-index subgroups of $G$, and an edge is drawn between two vertices if one is a $p$-power index subgroup of the other.
We set $cd_p(G)$ to be the connected diameter of the $p$-containment graph of $G$.

\begin{lemma} \label{lem:specgeod}
$cd_3(\Sym_4) = 4$.
Moreover, any path in the $3$-containment graph connecting $\left< (1, 2) \right>$ to $\left< (3,4) \right>$ must have two consecutive vertices $\Delta_1$ and $\Delta_2$ such that $(i,j) \in \Delta_1 \cap \Delta_2$ where $i \in \{ 1, 2 \}$ and $j \in \{ 3, 4 \}$.
\end{lemma}

\begin{proof}
We discovered this through \cite{GAP4} computations. Please see Figure \ref{fig:conngraphs4}, the component containing $\left< (1, 2) \right>$ is in the upper left corner. 
The geodesics connecting $\left< (1, 2) \right>$ to $\left< (3,4) \right>$ are of the form
$$
\left< (1, 2) \right>, \left< (1,2) , (1,2,k) \right>, \left< (i,k) \right>, \left< (i,k), (i, j, k) \right>, \left< (3,4) \right>,
$$
where $i \in \{1,2\}$ and $j, k \in \{ 3, 4 \}$.
Removing back-tracking from any path connecting $\left< (1, 2) \right>$ to $\left< (3,4) \right>$ reduces to one of the above geodesics. The lemma follows immediately.
\end{proof}

\begin{figure} \label{fig:conngraphs4}
\begin{center}
\includegraphics[scale=0.5]{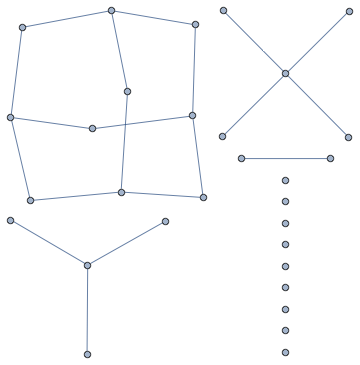}
\caption{The $3$-containment graph for $\Sym_4$. The component containing $\left< (1, 2) \right>$ is in the upper left corner.  The vertices $\left<(1,2)\right>$ and $\left< (3,4) \right>$ are the two vertices inside the square. This image was created with the use of \cite{MATHEMATICA}.}
\end{center}
\end{figure}

We remark that the corresponding statement in Lemma \ref{lem:construction}, below, for connected diameters of $p$-local commensurability graphs is \emph{not} true. This is why we use $p$-containment graphs.

\begin{lemma} \label{lem:construction}
Let $H_1, H_2, H_3$ and $H_4$ be isomorphic copies of a group $H$ and set
$\Delta = H_1 \times H_2 \times H_3 \times H_4$ and 
$$
G =  \Delta \rtimes \Sym_4,
$$
where the action of $\Sym_4$ on $\Delta$ is the permutation action on the coordinates.
Then
$$
cd_3(G) \geq cd_3(H) + 1.
$$
\end{lemma} 

\begin{proof}
Let $V_1, V_2, \ldots, V_d$ be a geodesic path in the $3$-containment graph of $H$ of length  $d := cd_3(H)$.
Set
$$
E_1 = \left< V_1 \times V_1 \times V_d \times V_d , (1, 2) \right>
$$
and
$$
E_2 = \left< V_d \times V_d \times V_1 \times V_1 , (3,4) \right>.
$$

These two vertices are connected via a path given by the vertices:
\begin{eqnarray*}
E_1 =  \left< V_1 \times V_1 \times V_d \times V_d , (1, 2) \right>, \\
\left< V_{2} \times V_{2} \times V_d \times V_d, (1,2) \right>, \\
\left< V_{3} \times V_{3} \times V_d \times V_d, (1,2) \right>, \\
\vdots \\
\left< V_{d} \times V_{d} \times V_d \times V_d, (1,2) \right>, \\
\left< V_{d} \times V_{d} \times V_d \times V_d, (1,2,3), (1,2) \right> ,\\
\left< V_{d} \times V_{d} \times V_d \times V_d, (2,3) \right>, \\
\left< V_{d} \times V_{d} \times V_d \times V_d, (2,3,4), (2,3)) \right>, \\
\left< V_{d} \times V_{d} \times V_d \times V_d, (3,4) \right>, \\
\left< V_{d} \times V_{d} \times V_{d-1} \times V_{d-1}, (3,4) \right>, \\
\left< V_{d} \times V_{d} \times V_{d-2} \times V_{d-2}, (3,4) \right>, \\
\vdots \\
\left< V_{d} \times V_{d} \times V_1 \times V_1, (3,4) \right>.
\end{eqnarray*}
Thus, they are in the same component. To finish, suppose, for the sake contradiction, that there is a geodesic with endpoints $E_1$ and $E_2$ with length equal to $d$: $E_1 = A_1, A_2, \ldots, A_d = E_2$.
Let $\pi$ be the projection $G \to G/\Delta \cong \Sym_4$ and $\phi_i$ be the maps $G \to H_i$ given by $A \mapsto A \cap \Delta \mapsto \pi_i(A)$, where $\pi_i : \Delta \to H_i$ is the projection map.

Using Lemma \ref{lem:padjcap} and \cite[Lemma 5]{BD15}, it is straightforward to see that $\phi_i(A_i)$ and $\pi(A_i)$ are both paths in the $3$-containment graphs of $H$ and $\Sym_4$, respectively. Moreover, if $\pi(A_i)$ contains $(j,k)$, then $\phi_j(A_i) = \phi_k(A_i)$ must be equal.
By Lemma \ref{lem:specgeod}, there must be $i$ where $\pi(A_i)$ and $\pi(A_{i+1})$ contains some $(j,k)$ where $j \in \{ 1, 2 \}$ and $k \in \{ 3, 4 \}$, and hence for some $\ell$, $\phi_\ell(A_1), \ldots, \phi_\ell(A_d)$ has a repeated vertex, which is impossible as this path must be a geodesic of length $d$.
\end{proof}

For a group $H$, we denote the group constructed in the above lemma by $BS(H)$.

\begin{lemma} \label{lem:cd}
Let $n$ be a natural number.
If $cd_3(G) > 2n$, then there exists a component of the $3$-local commensurability graph of $G$ with graph diameter at least $n$.
\end{lemma}

\begin{proof}
Let $L$ be a geodesic in the $3$-containment graph of length $2n$ with endpoints $E_1$ and $E_2$.
It is straightforward to see that $L$ is in a connected component of the $3$-local commensurability graph. Let $J$ be a geodesic path in the $3$-local commensurability graph given by the vertices
$$
E_1 = V_1, V_2, \ldots, V_k = E_2.
$$
Then the vertices
$$
E_1 = V_1, V_1 \cap V_2, V_2, V_2 \cap V_3, V_3, \ldots , V_k = E_2.
$$
give a path of length $2k$ in the $3$-containment graph.
We must have, then, that $2k \geq 2n$, giving $k \geq n$.
So we have found a component of the $3$-local graph of graph diameter at least $n$.
\end{proof}

We are now ready to complete the proof.

\begin{proof}[Proof of Theorem \ref{thm:mainC}]
Let $n$ a natural number be given.
Consider the group
$$
BS(BS(BS( \cdots (BS(\Sym_4)\cdots),
$$
where the number of $BS$ applications is greater than $2n$.
By Lemmas \ref{lem:construction}, $cd_3$ of the resulting group is greater than $2n$.
Since $Sym_4$ is solvable, and any solvable-by-solvable group is solvable, it follows that the above group is solvable.
Hence, by Lemma \ref{lem:cd}, the proof is complete.
\end{proof}

\begin{proof}[Proof of Corollary \ref{cor:main}]
Let $N > 0$ be given.
By Theorem \ref{thm:mainC}, there exists a finite group $Q$ of with a component $\Omega$ of the $3$-local commensurability graph with graph diameter greater than $N$.
Any free group of rank $2$ contains, as a normal subgroup, a free group of rank equal to $|Q|$. Call this subgroup $\Delta$.
Then $\Delta$ maps onto $Q$, so we are done by \cite[Lemma 8]{BD15}.
\end{proof}

\appendix
\section{Some examples}

Here will give a complete classification of the components of $p$-local commensurability graphs of  the subgroup of upper triangular matrices, $P(2,q)$ of GL(2, $\mathbb{F}_{q}$), where $q$ is some prime. By Theorem \ref{thm:mainA} in this paper, the connected diameter of this graph is at most 4. Through our analysis, we will see that the connected diameter is at most 2 (and this is sharp).

Before stating our classification of components, we need a definition. Let $H$ be component of graph $\Gamma$. We say $H$ is a \emph{star graph} if every edge in $H$ is incident with a fixed vertex $v_c$ in $H$, and these are the only edges in $H$.
We call $v_c$ the \emph{ center} of $H$, and it is unique. With this definition in hand, we can now state the main result of this appendix:

\begin{theorem}
\label{thm:compclassp}
Let $p, q$ be prime numbers.
The components of $\Gamma_{p}(P(2,q))$ are complete or star.
\end{theorem}

\noindent

 We devote the rest of this section to proving Theorem \ref{thm:compclassp}.
We begin with primes $p$ that do not divide the order of the group.
 
\begin{proposition} \label{prop:totaldisc}
Let $G$ be a finite group.
Then for any prime $p$, $\Gamma_{p}(G)$ is totally disconnected if and only if $p\nmid |G|$.
 \end{proposition}

\begin{proof} 
Set $N = |G|$.
If $p |N$ then there exists a nontrivial $p$-Sylow subgroup which is $p$-adjacent to the trivial group.
Now, if $p\nmid N$. Suppose, for the sake of a contradiction, that there exist $A < B \leq G$ such that $A$ is of index a power of $p$ in $B$. Then $p$ divides $|G|$, which is impossible. It follows that $\Gamma_p(G)$ is totally disconnected.
\end{proof}

Note that by Proposition \ref{prop:totaldisc}, if $p$ does not appear in the prime factorization of order of $G$, then the $p$-commensurability graphs are totally disconnected. Thus, for the remainder, it suffices to handle the case when $p$ divides the order of $P(2, q)$.

\begin{center}
\begin{table}[ht]
\centering
\begin{tabular}{c  c | c c}
\includegraphics[scale=0.5]{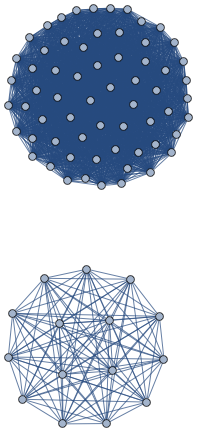} & & & \includegraphics[scale=0.5]{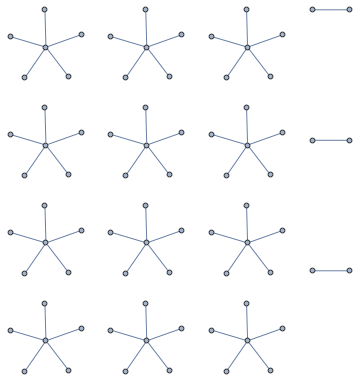} 
\end{tabular}
\caption{The two graphs above are $\Gamma_{2}(P(2,5))$ and $\Gamma_{5}(P(2,5))$, respectively. For $p=2$ we have two complete components, and for $p=5$ we have twelve star graphs. This image was created with the use of \cite{MATHEMATICA}.}
\end{table}
\end{center}

We will identify $P(2,q)$ with the semidirect product $\mathbb{Z}/{q}\rtimes (\mathbb{Z}/(q-1)\times \mathbb{Z}/(q-1))$, where $\mathbb{Z}/q$ corresponds to the subgroup $\left< \begin{pmatrix} 1 & 1  \\ 0 & 1 \end{pmatrix} \right>$ and $\mathbb{Z}/(q-1) \times \mathbb{Z}/(q-1)$ corresponds to diagonal matrices in $\GL(n, \mathbb{F}_q)$.

\begin{proposition} \label{prop:subgroupclass}
All the subgroups of $P(2,q) \cong \mathbb{Z}/{q}\rtimes (\mathbb{Z}/(q-1)\times \mathbb{Z}/(q-1))$ are the form of $\left<\mathbb{Z}/{q},s_{i}\right>$ and $\left<b^{n}s_{i}b^{-n}\right>$, where $\{ s_{j} \}$ enumerates all subgroups of $\mathbb{Z}/(q-1)\times \mathbb{Z}/(q-1)$, and
 $b$ is the generator of $\mathbb{Z}/{q}$, $n\in \mathbb{Z}/{q}$.
 \end{proposition}
\begin{proof}
Let $G=\mathbb{Z}/{q}\rtimes (\mathbb{Z}/(q-1)\times \mathbb{Z}/(q-1))$ and $U$ be a subgroup of $G$. 
Also, let $\mathbb{Z}/{q}=\left<b\right>$, $b$ be the generator of $\mathbb{Z}/{q}$.
Let $P$ be the $q$-Sylow subgroup of $U$.
Then $\mathbb{Z}/{q}$ is the unique $q$-Sylow subgroup of $\mathbb{Z}/{q}\rtimes(\mathbb{Z}/(q-1)\times \mathbb{Z}/(q-1))$.
So either $P = \mathbb{Z}/{q}$ or $P = \{ e \}$.

Then by Hall's theorem, we have a Sylow system for $U$, 
$$S_{p_{1}},S_{p_{2}},\dots,S_{p_{k}},$$ 
$|U|=p_{1}^{l_{1}}p_{2}^{l_{2}}\dots p_{k}^{l_{k}}$, where $S_{p_{i}}S_{p_{j}}=S_{p_{j}}S_{p_{i}}$. 
Let $p_{1}=p$, so $S_{p}=P$. Consider $S_{p_{2}}S_{p_{3}}\dots S_{p_{k}}$, then $|S_{p_{2}}S_{p_{3}}\dots S_{p_{k}}|=\frac{|U|}{|P|}$. That is $|S_{p_{2}}S_{p_{3}}\dots S_{p_{k}}|$ is not divisible by $p$. Consider, $\Delta=\mathbb{Z}/(q-1)\times \mathbb{Z}/(q-1)$, then $\Delta$ is a Hall-subgroup. Then again by Hall's Theorem, there exist $\gamma\in G$, such that 
\begin{equation*}
\gamma S_{p_{2}}S_{p_{3}}\dots S_{p_{k}} \gamma^{-1}<\Delta.
\end{equation*}
So we have $U=\left<P, \gamma S(i) \gamma^{-1}\right>$, where $S(i)=S_{p_{2}}S_{p_{3}}\dots S_{p_{k}}$ and we index all the subgroups of $\mathbb{Z}/(q-1)\times \mathbb{Z}/(q-1)$ by $i\in I$ some index set.

Since $P$ is either the identity or the entire $\mathbb{Z}/{q}$, and $\mathbb{Z}/{q}$ is normal in $G$, then the result follows.
\end{proof}

Before we proceed, we need a quick technical result.

\begin{lemma} \label{lem:claim}
Let $\Delta_{1}$, $\Delta_{2}$ be subgroup of some finite group $\Gamma$ such that $\Delta_{1}$ and $\Delta_{2}$ are both nilpotent and $\Delta_{1}$ is $p$-adjacent to $\Delta_{2}$. Then $\Delta=\Delta_{1}\cap\Delta_{2}$ contains the normal complement of the $p$-Sylow subgroup of $\Delta_{1}$
\end{lemma}

\begin{proof}
Let $\Delta_{1}$, $\Delta_{2}$ by two subgroup of a finite group and they both are nilpotent.
Since $\Delta_1$ is nilpotent, the normal complement of the $p$-Sylow subgroup of $\Delta_1$ exists, call it $N$.
By Hall's Theorem, the normal complement in $\Delta_1 \cap \Delta_2$ of the $p$-Sylow subgroup of $\Delta_1 \cap \Delta_2$ is conjugate to a subgroup of $N$, and hence is a subgroup of $N$.
Since elements from different $p$-Sylow subgroups commute with one another, we have that $N \Delta_1 \cap \Delta_2$ is a proper subgroup of $\Delta_1$ containing $\Delta_1 \cap \Delta_2$ as a subgroup of index \emph{not} a power of $p$.
Hence, it is impossible for $[\Delta_1 : \Delta_1 \cap \Delta_2] = [\Delta_1 : N (\Delta_1 \cap \Delta_2)][N (\Delta_1 \cap \Delta_2) : \Delta_1 \cap \Delta_2]$ to be a power of $p$.
\end{proof}

\begin{lemma} \label{lem:completecomp}
 Let $p$ be some prime. If $p\mid |\mathbb{Z}/(q-1)\times \mathbb{Z}/(q-1)|,$ then each component of $\Gamma_{p}(P(2,q))$ is complete.
 \end{lemma}

\begin{proof}
Let $\{ s_{i} \}_{i \in I}$ be an enumeration of all the subgroups of $\mathbb{Z}/(q-1)\times \mathbb{Z}/(q-1)$. Let $b$ be the generator of $\mathbb{Z}/{q}$, and $n\in \mathbb{Z}/{q}$.
It is straightforward to see that $[\left<\mathbb{Z}/{p},s_{i}\right>:\left<\mathbb{Z}/{p},s_{i}\cap s_{j}\right>]=[s_{i}:s_{i}\cap s_{j}]$, for all $i,j$.
%Note that
%$
%[\left<\mathbb{Z}/{p}:s_{i}\right>:\left<\mathbb{Z}/{p}:s_{i}\cap s_{j}\right>]=\frac{|\left<\mathbb{Z}/{q},s_{i}\right>|}{|\left<\mathbb{Z}/{q},s_{i}\cap s_{j}\right>|}.
%$
%Also, $|\left<\mathbb{Z}/{q},s_{i}\right>|=|\mathbb{Z}/{q}||s_{i}|$ and $|\left<\mathbb{Z}/{q},s_{i}\cap s_{j}\right>|=|\mathbb{Z}/{q}||s_{i}\cap s_{j}|$. Then the result follow.
Moreover, we claim that $\left<\mathbb{Z}/{q},s_{i}\right>$ is never $p$-adjacent to $\left<b^{n}s_{j}b^{-n}\right>$, for all $i,j$.
Indeed, suppose that $\left<\mathbb{Z}/{q},s_{i}\right>$ is $p$-adjacent to $\left<b^{n}s_{j}b^{-n}\right>$. It is clear the $q$ does not divides $|\left<b^{n}s_{j}b^{n}\right>|$, for all $j\in I$.
Note, 
\begin{equation*}
[P(2,q):\left<\mathbb{Z}/{q},s_{i}\right>]=\frac{q(q-1)^{2}}{q|s_{i}|}=\frac{(q-1)^{2}}{|s_{i}|}.
\end{equation*}
Since $q$ does not divide both $(q-1)^{2}$, then $q$ does not divide $\frac{(q-1)^{2}}{|s_{i}|}$. Therefore by \cite[Lemma 9]{BD15}, we will have that
\begin{equation*}
q\nmid \frac{q(q-1)^{2}}{|\left<b^{n}s_{i}b^{n}\right>|}
\end{equation*}
 Which is impossible because $q$ does not divide $|\left<b^{n}s_{j}b^{-n}\right>|$.
We conclude that the two types of subgroups appearing in Proposition \ref{prop:subgroupclass} never appear in the same component.

By our earlier claim we know that $[\left<\mathbb{Z}/{p}:s_{i}\right>:\left<\mathbb{Z}/{p}:s_{i}\cap s_{j}\right>]=[s_{i}:s_{i}\cap s_{j}]$, for all $i,j$. This implies that $\left<\mathbb{Z}/{q},s_{i}\right>$ is $p$-adjacent to $\left<\mathbb{Z}/{q},s_{j}\right>$ if and only if $s_{i}$ is $p$-adjacent to $s_{j}$. We know that all the $s_{i}$ are subgroup of $\mathbb{Z}/(q-1)\times \mathbb{Z}/(q-1)$ which is an abelian group, then by \cite[Theorem 1]{BD15} we know that the $p$-local commensurability graph are complete, then the component for $\left<\mathbb{Z}/{q},s_{i}\right>$ is also complete.

To finish, we will show that all components containing subgroups of the form $\left<b^{n}s_{i}b^{-n}\right>$ are also complete.
We break up the proof into two cases depending on whether $\left<b^{n}s_{i}b^{n}\right>$ has cardinality that is some power of $p$.

{\bf Case 1}: $|\left<b^{n}s_{i}b^{n}\right>|=|s_{i}|$ is some power of $p$.
It is clear that for all $s_{i_{1}},\dots s_{i_{k}}$ with order some power of $p$, we have $\left<b^{n}s_{i_{\alpha}}b^{-n}\right>$ is $p$-adjacent to $\left<b^{m}s_{i_{\beta}}b^{-m}\right>$, for $i_{1}\leq i_{\alpha},i_{\beta}\leq i_{k}$ and any $m,n$.
Now, let $|s_{i}|=p^{k}$ and $|s_{j}|=p^{m}a$, were $\gcd(p^{m},a)=1$ and $a\neq 1$.
Then $|\left<b^{n}s_{i}b^{-n}\right>|=p^{k}$ and $|\left<b^{v}s_{j}b^{v}\right>|=p^{m}a$. Thus, we have 
\begin{equation*}
|\left<b^{n}s_{i}b^{-n}\right>\cap\left<b^{v}s_{j}b^{v}\right>|=p^{l}
\end{equation*}
where $l\leq k,m$.
It follows from this that $A$ is not $p$-adjacent to $B$ for all $i,j,n,v$.
Therefore, by above argument we know that all the subgroups $\left<b^{n}s_{i}b^{-n}\right>$ with order some power of $p$ form a complete component.

{\bf Case 2}: Let $|s_{i}|$ not be a power of $p$, that is $|s_{i}|=p^{n}a$, where $\gcd(p^{n},a)=1$ and $a\neq 1$.
Then we may write
$$
s_i = H \times A,
$$
where $H$ is a $p$-Sylow subgroup and $|A| = a$.
It is clear that $\left<b^{n}s_{i}b^{-n}\right>$ is $p$-adjacent to $\left<b^{n}Ab^{-n}\right>$. That is, for any $n$ and $\left<b^{n}s_{i}b^{-n}\right>$ whose order is not a power of $p$, then there exist a subgroup $A_{n,i}$ of $s_{i}$ such that $\left<b^{n}s_{i}b^{-n}\right>$ is $p$-adjacent to $b^{n}A_{n,i} b^{-n}$.

%By above argument we see that for any subgroup $\left<b^{n}s_{i}b^{-n}\right>$ there exist a complement of the normal $p$-Sylow subgroup and they are connected. 
If there exist another group $A^{'}$ which is $p$-adjacent to $b^n A_{n,i} b^{-n}$, then it is easy to see that $A^{'}$ is $p$-adjacent to $\left<b^{n}s_{i}b^{-n}\right>$.
Similarly, if there is another group $B=\left<b^{m}s_{j}b^{-m}\right>$ which is $p$-adjacent to any $A^{'}$ as above, then by Lemma \ref{lem:claim} we have $B$ is $p$-adjacent to $b^n A_{n,i} b^{-n}$, and hence is $p$-adjacent to $\left<b^{n}s_{i}b^{-n}\right>$. It follows that the component containing $\left<b^{n}s_{i}b^{-n}\right>$ is complete, as desired.
\end{proof}

%Our next is Lemma 13 in paper \cite{BD15}. 
%
%\begin{lemma}
% If $Q$ is a finite solvable group that is not nilpotent then there is some prime $p$ so that a connected component of $\Gamma_{p}(Q)$ is not complete.
% \end{lemma}

We are left with $p = q$, so the proof of Theorem \ref{thm:compclassp} is completed with the next result.

\begin{proposition} 
Each component of $\Gamma_{p}(P(2,p))$ is
\begin{enumerate}
\item a star graph, or 
\item a complete graph with two vertices.
\end{enumerate}
\end{proposition}

\begin{proof} 
By Proposition \ref{prop:subgroupclass}, we need only consider subgroups of the form $\left<\mathbb{Z}/{p},s_{i}\right>$ and $\left<b^{-n}s_{j}b^{n}\right>$.

We give a brief outline.
We first show that any two subgroups $\left<b^{-n}s_{i}b^{n}\right>$ and $\left<b^{-m}s_{j}b^{m}\right>$ are not $p$-adjacent if $i \neq j$ or $n \neq m$. Clearly, for any $i$ and $n$, $\left<\mathbb{Z}/{p},s_{i}\right>$ adjacent to $\left<b^{-n}s_{i}b^{n}\right>$. So, next, we will show that $\left<\mathbb{Z}/{p},s_{i}\right>$ is not $p$-adjacent to any $\left<b^{-n}s_{j}b^{n}\right>$ for $i \neq j$. Finally, we show that there is only one group of the form $\left<\mathbb{Z}/{p},s_{j}\right>$ that is adjacent to $\left<\mathbb{Z}/{p},s_{i}\right>$ for any fixed $i$.  Thus, all the components under consideration are star graphs with centers $\left< \mathbb{Z}/{p},s_{i}\right>$, or are complete graphs with two vertices.

Let $A_{n}=\left<b^{-n}s_{i}b^{n}\right>$ and $A_{m}=\left<b^{-m}s_{j}b^{m}\right>$, $A_{n}\neq A_{m}$. Let $p$ be a prime such that $p\nmid (q-1)^{2}$.
So if $A_{n}$ and $A_{m}$ are adjacent, then $[A_{n}:A_{n}\cap A_{m}]=p^{c}$,  so $c = 0$. By symmetry, $A_{n}$ is not adjacent to $A_{m}$ since $A_n \neq A_m$.

Now, as we note in the proof of Lemma \ref{lem:completecomp}, $\left<\mathbb{Z}/{p},s_{i}\right>$ is $p$-adjacent to $\left<\mathbb{Z}/{p},s_{j}\right>$ if and only if $s_{i}$ is $p$-adjacent to $s_{j}$. Since $q$ does not divide $(q-1)^{2}$ and every $s_{i}$, $s_{j}$ has order which divide $(q-1)^{2}$, then we have $s_{i}$ is not $p$-adjacent to $s_{j}$ for all $i,j$. Hence $\left<\mathbb{Z}/{p},s_{i}\right>$ is not $p$-adjacent to $\left<\mathbb{Z}/{p},s_{j}\right>$.

Finally, we will show that any group of the form $\left<\mathbb{Z}/{p},s_{i}\right>$ that is adjacent to $A_{n}$ is unique.
Clearly, $\left<\mathbb{Z}/{p},s_{i}\right>$ is $p$-adjacent to $\left<b^{n}s_{i}b^{-n}\right>$. Now, if $\left<\mathbb{Z}/{p},s_{j}\right>$ is also $p$-adjacent to $\left<b^{n}s_{i}b^{-n}\right>$, then we have 
\begin{equation*}
[\left<b^{n}s_{i}b^{-n}\right>:\left<\mathbb{Z}/{p},s_{j}\right>\cap\left<b^{n}s_{i}b^{-n}\right>]=p^{k}.
\end{equation*}
Since $p$ does not divides the order of $\left<b^{n}s_{i}b^{-n}\right>$, then we must have $k=0$, and so $\left<b^{n}s_{i}b^{-n}\right> \leq \left<\mathbb{Z}/{p},s_{j}\right>$.
Moreover, since $\left<\mathbb{Z}/{p},s_{j}\right>$ is $p$-adjacent to $\left<b^{n}s_{i}b^{-n}\right>$,
\begin{equation*}
[\left<\mathbb{Z}/{p},s_{j}\right>:\left<b^{n}s_{i}b^{-n}\right>]=p^{k},
\end{equation*}
which gives $k = 1$, since $p$ does not divide $|s_j|$.
Thus, $\left<\mathbb{Z}/{p},s_{j}\right> = \left< b, b^n s_i b^{-n} \right> = \left<\mathbb{Z}/{p},s_{i}\right>$, as desired.
\end{proof}

\bibliography{refs}
\bibliographystyle{amsalpha}

\end{document}